\documentclass[12pt]{amsart}
\usepackage{latexsym, amsmath, amssymb,srcltx}
\usepackage{epsfig}
\usepackage{amscd}
\usepackage{pstricks,pst-node,cite}

\newtheorem{thm}{Theorem}[section]
\newtheorem{lem}[thm]{Lemma}

\newtheorem{cor}[thm]{Corollary}
\newtheorem{exa}[thm]{Example}

\newtheorem{remark}[thm]{Remark}
\theoremstyle{remark}

\newcommand{\B}[1]{\mathbb #1}

\newcommand{\BZ}{\mathbb{Z}}

\ifx\blackandwhite\undefined
  \newrgbcolor{darkred}{.9 0 0}\newrgbcolor{emgreen}{0 .9 0}
  \newrgbcolor{pup}{.7  0 .9}\newrgbcolor{xxx}{.8 .8 .8}
  \newrgbcolor{lightgray}{.99 .99 0}\newrgbcolor{darkgray}{.8 .8 .3} \else
\fi
\psset{linecolor=blue}
\psset{linewidth=1pt,dash=3pt 3pt,doublesep=.05,dotsize=1pt 5}
\SpecialCoor

\newcommand{\middlearrow}{\lput{:U}{\begin{pspicture}[shift=0](0,0)(0,0)
\psline[arrows=->,arrowscale=1.5](2.2pt,0)(2.3pt,0)\end{pspicture}}}

\begin{document}

\author{Hunki Baek}
\address{Department of Mathematics Education\\ Catholic University of Daegu\\
Gyeongsan, 712-702 Republic of Korea}
\email{hkbaek@cu.ac.kr}

\author{Sejeong Bang}
\address{Department of Mathematics \\Yeungnam University \\ Gyeongsan, 712-749 Republic of Korea}
\email{sjbang@ynu.ac.kr}
\thanks{This research was supported by Basic Science Research Program through the National Research Foundation of Korea(NRF) funded by the Ministry of Education,
Science and Technology (NRF-2011-0013985).}

\author{Dongseok Kim}
\address{Department of Mathematics \\Kyonggi University
\\ Suwon, 443-760 Republic of Korea}
\email{dongseok@kgu.ac.kr}
\thanks{}

\author{Jaeun Lee}
\address{Department of Mathematics \\Yeungnam University \\Gyeongsan, 712-749 Republic of Korea}
\email{julee@yu.ac.kr}
\thanks{}


\title[A bijection between aperiodic palindromes and connected circulant graphs]{A bijection between aperiodic palindromes and connected circulant graphs}

\begin{abstract}
In this paper, we show that there is a one-to-one correspondence between the set of compositions (resp. prime compositions) of $n$ and the set of circulant digraphs (resp.  connected circulant digraphs) of order $n$. We also show that there is a one-to-one correspondence between the set of palindromes (resp. aperiodic palindromes) of $n$ and the set of circulant graphs (resp. connected circulant graphs) of order $n$. As a corollary of this correspondence, we enumerate the number of connected circulant (di)graphs of order $n$.
\end{abstract}

\maketitle

\section{Introduction}
A \emph{composition $\sigma=\sigma_1\sigma_2 \ldots\sigma_m$} of $n$
is an ordered word of one or more positive integers
whose sum is $n$. The number of summands $m$ is called the \emph{number of parts} of the
composition $\sigma$. A composition of $n$ without order gives a partition of $n$. A composition of $n$ with $k$ parts is \emph{aperiodic}
if its period is $k$. In other words, a composition is aperiodic if it is not the concatenation of a proper part of given composition. A composition $\sigma=\sigma_1\sigma_2 \ldots\sigma_m$ is called \emph{prime} if $\gcd(\sigma)=\gcd\{\sigma_1,\ldots,\sigma_m\}=1$. A \emph{numeral palindrome} (or simply, \emph{palindrome}) is a composition in which the summands in given order are the same with those in the reverse order ($i.e.$, $\sigma=\sigma^{-1}$ where $\sigma^{-1}=\sigma_m\ldots \sigma_{2}\sigma_1$ for $\sigma=\sigma_1\sigma_2 \ldots\sigma_m $). It is known in \cite{HCG} that the number of palindromes of $n\geq 2$ is $2^{\lfloor\frac{n}{2} \rfloor}$. The first $30$ palindromes in decimal can be found as the
sequence $A002113$ in OEIS~\cite{OEIS}. Several types of palindromes are studied (see \cite{OEIS}, \cite{wiki}). The number of aperiodic palindromes of $n$ with $k$ parts ($1 \le k\le n$) is studied but the numbers are known only for $n\le 55$ (see OEIS~\cite{OEIS}). However it is still unknown whether there exist infinitely many palindromic primes or not, where a palindromic prime is a positive integer which is prime and also a palindrome. Although palindromes are often considered in the decimal system, the concept of palindromicity can be generalized to the natural numbers in any numeral system. An integer $m>0$ is called \emph{palindromic} in base $b \ge 2$ if it is written in standard notation 
$$m=\sum _{{i=0}}^{k}a_{i}b^{i}~~~~(a_k \ne 0)$$
with $k+1$ digits $a_i~(0\leq i\leq k)$ satisfying $a_i = a_{k-i}$ and $0 \le a_i < b$ for all $0\leq i\leq k$. \\

A \emph{circulant graph} is a graph whose automorphism group
includes a cyclic subgroup which acts transitively on
the vertex set of the graph. For a subset $S\subseteq \BZ_n$ satisfying $S = -S$ $\mod n$, a circulant graph of order $n$ denoted by $G(n,S)$ is a graph with vertex set $\{0, 1, \ldots , n-1\}$ and edge set $E$, where $\{i,j\} \in E$ if and
only if $i\neq j$ and $j-i \in S \mod n$. A \emph{circulant digraph} is also defined without the condition $S=-S$. That is, for a subset $S\subseteq \BZ_n$, a circulant digraph $G(n,S)$ is a digraph with
vertex set $\{0, 1, \ldots , n-1\}$ and arc set $A$, where $(i,j) \in A$ if and
only if $i\neq j$ and $j-i \in S\mod n$.\\ 
\indent Isomorphism problem of circulant graphs had been studied by several authors~(see \cite{AP, Dobson}) and it is completely
solved by Muzychuk~\cite{Muzy2}. In \cite{KKL}, Kim, Kwon and Lee found degree distribution polynomials for the equivalence classes of circulant graphs of several types of order. They also found an enumeration formula for the number of equivalence classes
of circulant graphs and they listed in~\cite[Table 1]{KKL} the degree distribution polynomials
$\Psi_{\BZ_{n}}(x)$ and the number of equivalence classes of circulant 
graphs $\mathcal{E}(\BZ_n)$ for $1\leq n \le 20$. We observe from \cite[Table 1]{KKL} that the number of equivalence classes of circulant 
graphs $\mathcal{E}(\BZ_n)$ is equal to the number of aperiodic palindromes of $n$ for $1\leq n\leq 20$. This leads us to study a coincidence between the set of circulant graphs of order $n$ and the set of aperiodic palindromes of $n$. In this paper, we study the coincidence and we extend this to circulant digraphs of order $n$ and compositions of $n$.   \\

The outline of this paper is as follows. In Section~\ref{correspondence}, we show that there is a one-to-one correspondence between the set of compositions of $n$ and the set of circulant digraphs of order $n$ (see Theorem \ref{dg&c}). In particular, we also show that this bijection in Theorem \ref{dg&c} guarantees a one-to-one correspondence between the set of prime compositions of $n$ and the set of connected circulant digraphs of order $n$ (see Theorem \ref{dicirculantthm}). As an application, we enumerate the number of connected circulant digraphs (i.e., the number of prime compositions), disconnected circulant digraphs and
circulant digraphs of outdegree $k$ (see Corollary \ref{ancounting}). In Section~\ref{undirected}, we first show that there is a one-to-one correspondence between the set of palindromes of $n$ and the set of circulant graphs of order $n$ (see Theorem \ref{palthm}). In particular, we also show that this bijection in Theorem \ref{palthm} guarantees a one-to-one correspondence between the set of aperiodic palindromes of $n$ and the set of connected circulant graphs of order $n$ (see Theorem \ref{apn-thm}).  As a corollary, we give an enumeration formula of the number of aperiodic palindromes of $n$ (i.e., the number of connected circulant graphs of order $n$) in Corollary \ref{acounting}.

\section{One-to-one correspondence between the circulant digraphs of order $n$ and the compositions of $n$} \label{correspondence}
In this section, we first show that there is a bijection between the set of compositions of $n$ and the set of circulant digraphs of order $n$ in Theorem \ref{dg&c}. In particular, we also show that this bijection in Theorem \ref{dg&c} guarantees a bijection between the set of prime compositions of $n$ and the set of connected circulant digraphs of order $n$ (see Theorem \ref{dicirculantthm}). Moreover, we will enumerate the number of connected circulant digraphs of order $n$ (i.e., the number of prime compositions of $n$) by using the M\"obius inversion formula (see Corollary \ref{ancounting}).\\

For each integer $n\geq 1$, we define
$$ C(n) = \{ \sigma \mid \sigma~{\rm{is}~\rm{a}~\rm{composition}~\rm{of}}~n\} \mbox{~and~} \mathcal{G}_C(n) = \{ \Omega \mid \Omega \subseteq \BZ_n, 0 \in \Omega \} .$$
Without loss of generality, each element $\Omega \in \mathcal{G}_C(n)$ with $|\Omega|=t\geq 1$ is denoted by $\Omega=\{ a_1, a_2,\ldots, a_t\}$ with $0=a_1<a_2<\cdots<a_t$.

Note here that in view of \cite[Theorem 2.2]{KKL}, there is a bijection between the set  $\mathcal{G}_C(n)$ and the set of equivalence classes of circulant digraphs of order $n$. In Table~\ref{com5exafig}, we consider the elements $\Omega$ of $\mathcal{G}_C(5)$, the corresponding circulant digraphs $G(5,\Omega)$ of order $5$ and the corresponding compositions $\sigma_{\Omega}$ of $5$ which will be defined in (\ref{sigmaS}).\\

\begin{table}
\begin{tabular}{||c|c|c||c|c|c||}\hline
$\Omega \in \mathcal{G}_C(5)$ & $G(5,\Omega)$ & $\sigma_{\Omega}\in C(5)$ & $\Omega \in \mathcal{G}_C(5)$ & $G(5,\Omega)$ & $\sigma_{\Omega}\in C(5)$ \\ \hline
$\{0\}$ &
$
\begin{pspicture}[shift=-1](-1.4,-1.1)(1.4,1.3)
\pscircle[fillcolor=lightgray, fillstyle=solid, linewidth=1pt](1;18){.2}
\pscircle[fillcolor=lightgray, fillstyle=solid, linewidth=1pt](1;90){.2}
\pscircle[fillcolor=lightgray, fillstyle=solid, linewidth=1pt](1;162){.2}
\pscircle[fillcolor=lightgray, fillstyle=solid, linewidth=1pt](1;234){.2}
\pscircle[fillcolor=lightgray, fillstyle=solid, linewidth=1pt](1;306){.2}
\rput(1;18){{\footnotesize{0}}}\rput(1;90){{{\footnotesize{1}}}}\rput(1;162){{{\footnotesize{2}}}}
\rput(1;234){{\footnotesize{3}}}\rput(1;306){{\footnotesize{4}}}
\end{pspicture}$ & $5$ & $\{0,2,3\}$  &
$\begin{pspicture}[shift=-1](0,-1.1)(0,1.3) \end{pspicture}
\begin{pspicture}[shift=-1](0,1)(0,1)
\begin{pspicture}[shift=-1](-1.4,-1.1)(1.4,1.3)
\pccurve[angleA=160,angleB=20,ncurv=1](1;18)(1;162)\middlearrow
\pccurve[angleA=232,angleB=92,ncurv=1](1;90)(1;234)\middlearrow
\pccurve[angleA=304,angleB=164,ncurv=1](1;162)(1;306)\middlearrow
\pccurve[angleA=26,angleB=246,ncurv=1](1;234)(1;18)\middlearrow
\pccurve[angleA=88,angleB=308,ncurv=1](1;306)(1;90)\middlearrow
\pcline(1;18)(1;234)\middlearrow \pcline(1;234)(1;90)\middlearrow
\pcline(1;90)(1;306)\middlearrow \pcline(1;306)(1;162)\middlearrow
\pcline(1;162)(1;18)\middlearrow
\pscircle[fillcolor=lightgray, fillstyle=solid, linewidth=1pt](1;18){.2}
\pscircle[fillcolor=lightgray, fillstyle=solid, linewidth=1pt](1;90){.2}
\pscircle[fillcolor=lightgray, fillstyle=solid, linewidth=1pt](1;162){.2}
\pscircle[fillcolor=lightgray, fillstyle=solid, linewidth=1pt](1;234){.2}
\pscircle[fillcolor=lightgray, fillstyle=solid, linewidth=1pt](1;306){.2}
\rput(1;18){{\footnotesize{0}}}\rput(1;90){{{\footnotesize{1}}}}\rput(1;162){{{\footnotesize{2}}}}
\rput(1;234){{\footnotesize{3}}}\rput(1;306){{\footnotesize{4}}}
\end{pspicture}\end{pspicture}$  & $212$ \\ \hline
$\{0,1\}$ &
$\begin{pspicture}[shift=-1](0,-1.1)(0,1.3) \end{pspicture}
\begin{pspicture}[shift=-1](0,1)(0,1)\begin{pspicture}[shift=-1](-1.4,-1.1)(1.4,1.3)
\pcline(1;18)(1;90)\middlearrow \pcline(1;90)(1;162)\middlearrow
\pcline(1;162)(1;234)\middlearrow \pcline(1;234)(1;306)\middlearrow
\pcline(1;306)(1;18)\middlearrow
\pscircle[fillcolor=lightgray, fillstyle=solid, linewidth=1pt](1;18){.2}
\pscircle[fillcolor=lightgray, fillstyle=solid, linewidth=1pt](1;90){.2}
\pscircle[fillcolor=lightgray, fillstyle=solid, linewidth=1pt](1;162){.2}
\pscircle[fillcolor=lightgray, fillstyle=solid, linewidth=1pt](1;234){.2}
\pscircle[fillcolor=lightgray, fillstyle=solid, linewidth=1pt](1;306){.2}
\rput(1;18){{\footnotesize{0}}}\rput(1;90){{{\footnotesize{1}}}}\rput(1;162){{{\footnotesize{2}}}}
\rput(1;234){{\footnotesize{3}}}\rput(1;306){{\footnotesize{4}}}
\end{pspicture}\end{pspicture}$ & $14$ & $\{0,2,4\}$  &
$\begin{pspicture}[shift=-1](0,-1.1)(0,1.3) \end{pspicture}
\begin{pspicture}[shift=-1](0,1)(0,1)\begin{pspicture}[shift=-1](-1.4,-1.1)(1.4,1.3)
\pcline(1;18)(1;162)\middlearrow \pcline(1;162)(1;306)\middlearrow
\pcline(1;306)(1;90)\middlearrow \pcline(1;90)(1;234)\middlearrow
\pcline(1;234)(1;18)\middlearrow
\pcline(1;18)(1;306)\middlearrow \pcline(1;306)(1;234)\middlearrow
\pcline(1;234)(1;162)\middlearrow \pcline(1;162)(1;90)\middlearrow
\pcline(1;90)(1;18)\middlearrow
\pscircle[fillcolor=lightgray, fillstyle=solid, linewidth=1pt](1;18){.2}
\pscircle[fillcolor=lightgray, fillstyle=solid, linewidth=1pt](1;90){.2}
\pscircle[fillcolor=lightgray, fillstyle=solid, linewidth=1pt](1;162){.2}
\pscircle[fillcolor=lightgray, fillstyle=solid, linewidth=1pt](1;234){.2}
\pscircle[fillcolor=lightgray, fillstyle=solid, linewidth=1pt](1;306){.2}
\rput(1;18){{\footnotesize{0}}}\rput(1;90){{{\footnotesize{1}}}}\rput(1;162){{{\footnotesize{2}}}}
\rput(1;234){{\footnotesize{3}}}\rput(1;306){{\footnotesize{4}}}
\end{pspicture}\end{pspicture}$  & $221$ \\ \hline
$\{0,2\}$ &
$\begin{pspicture}[shift=-1](0,-1.1)(0,1.3) \end{pspicture}
\begin{pspicture}[shift=-1](0,1)(0,1)\begin{pspicture}[shift=-1](-1.4,-1.1)(1.4,1.3)
\pcline(1;18)(1;162)\middlearrow \pcline(1;162)(1;306)\middlearrow
\pcline(1;306)(1;90)\middlearrow \pcline(1;90)(1;234)\middlearrow
\pcline(1;234)(1;18)\middlearrow
\pscircle[fillcolor=lightgray, fillstyle=solid, linewidth=1pt](1;18){.2}
\pscircle[fillcolor=lightgray, fillstyle=solid, linewidth=1pt](1;90){.2}
\pscircle[fillcolor=lightgray, fillstyle=solid, linewidth=1pt](1;162){.2}
\pscircle[fillcolor=lightgray, fillstyle=solid, linewidth=1pt](1;234){.2}
\pscircle[fillcolor=lightgray, fillstyle=solid, linewidth=1pt](1;306){.2}
\rput(1;18){{\footnotesize{0}}}\rput(1;90){{{\footnotesize{1}}}}\rput(1;162){{{\footnotesize{2}}}}
\rput(1;234){{\footnotesize{3}}}\rput(1;306){{\footnotesize{4}}}
\end{pspicture}\end{pspicture}$ & $23$ & $\{0,3,4\}$  &
$\begin{pspicture}[shift=-1](0,-1.1)(0,1.3) \end{pspicture}
\begin{pspicture}[shift=-1](0,1)(0,1)\begin{pspicture}[shift=-1](-1.4,-1.1)(1.4,1.3)
\pcline(1;18)(1;234)\middlearrow \pcline(1;234)(1;90)\middlearrow
\pcline(1;90)(1;306)\middlearrow \pcline(1;306)(1;162)\middlearrow
\pcline(1;162)(1;18)\middlearrow
\pcline(1;18)(1;306)\middlearrow \pcline(1;306)(1;234)\middlearrow
\pcline(1;234)(1;162)\middlearrow \pcline(1;162)(1;90)\middlearrow
\pcline(1;90)(1;18)\middlearrow
\pscircle[fillcolor=lightgray, fillstyle=solid, linewidth=1pt](1;18){.2}
\pscircle[fillcolor=lightgray, fillstyle=solid, linewidth=1pt](1;90){.2}
\pscircle[fillcolor=lightgray, fillstyle=solid, linewidth=1pt](1;162){.2}
\pscircle[fillcolor=lightgray, fillstyle=solid, linewidth=1pt](1;234){.2}
\pscircle[fillcolor=lightgray, fillstyle=solid, linewidth=1pt](1;306){.2}
\rput(1;18){{\footnotesize{0}}}\rput(1;90){{{\footnotesize{1}}}}\rput(1;162){{{\footnotesize{2}}}}
\rput(1;234){{\footnotesize{3}}}\rput(1;306){{\footnotesize{4}}}
\end{pspicture}\end{pspicture}$  & $311$ \\ \hline
$\{0,3\}$ &
$\begin{pspicture}[shift=-1](0,-1.1)(0,1.3) \end{pspicture}
\begin{pspicture}[shift=-1](0,1)(0,1)\begin{pspicture}[shift=-1](-1.4,-1.1)(1.4,1.3)
\pcline(1;18)(1;234)\middlearrow \pcline(1;234)(1;90)\middlearrow
\pcline(1;90)(1;306)\middlearrow \pcline(1;306)(1;162)\middlearrow
\pcline(1;162)(1;18)\middlearrow
\pscircle[fillcolor=lightgray, fillstyle=solid, linewidth=1pt](1;18){.2}
\pscircle[fillcolor=lightgray, fillstyle=solid, linewidth=1pt](1;90){.2}
\pscircle[fillcolor=lightgray, fillstyle=solid, linewidth=1pt](1;162){.2}
\pscircle[fillcolor=lightgray, fillstyle=solid, linewidth=1pt](1;234){.2}
\pscircle[fillcolor=lightgray, fillstyle=solid, linewidth=1pt](1;306){.2}
\rput(1;18){{\footnotesize{0}}}\rput(1;90){{{\footnotesize{1}}}}\rput(1;162){{{\footnotesize{2}}}}
\rput(1;234){{\footnotesize{3}}}\rput(1;306){{\footnotesize{4}}}
\end{pspicture}\end{pspicture}$ & $32$ & $\{0,1,2,3\}$  &
$\begin{pspicture}[shift=-1](0,-1.1)(0,1.3) \end{pspicture}
\begin{pspicture}[shift=-1](0,1)(0,1)\begin{pspicture}[shift=-1](-1.4,-1.1)(1.4,1.3)
\pcline(1;18)(1;90)\middlearrow \pcline(1;90)(1;162)\middlearrow
\pcline(1;162)(1;234)\middlearrow \pcline(1;234)(1;306)\middlearrow
\pcline(1;306)(1;18)\middlearrow
\pccurve[angleA=160,angleB=20,ncurv=1](1;18)(1;162)\middlearrow
\pccurve[angleA=232,angleB=92,ncurv=1](1;90)(1;234)\middlearrow
\pccurve[angleA=304,angleB=164,ncurv=1](1;162)(1;306)\middlearrow
\pccurve[angleA=26,angleB=246,ncurv=1](1;234)(1;18)\middlearrow
\pccurve[angleA=88,angleB=308,ncurv=1](1;306)(1;90)\middlearrow
\pcline(1;18)(1;234)\middlearrow \pcline(1;234)(1;90)\middlearrow
\pcline(1;90)(1;306)\middlearrow \pcline(1;306)(1;162)\middlearrow
\pcline(1;162)(1;18)\middlearrow
\pscircle[fillcolor=lightgray, fillstyle=solid, linewidth=1pt](1;18){.2}
\pscircle[fillcolor=lightgray, fillstyle=solid, linewidth=1pt](1;90){.2}
\pscircle[fillcolor=lightgray, fillstyle=solid, linewidth=1pt](1;162){.2}
\pscircle[fillcolor=lightgray, fillstyle=solid, linewidth=1pt](1;234){.2}
\pscircle[fillcolor=lightgray, fillstyle=solid, linewidth=1pt](1;306){.2}
\rput(1;18){{\footnotesize{0}}}\rput(1;90){{{\footnotesize{1}}}}\rput(1;162){{{\footnotesize{2}}}}
\rput(1;234){{\footnotesize{3}}}\rput(1;306){{\footnotesize{4}}}
\end{pspicture}\end{pspicture}$  & $1112$ \\ \hline
$\{0,4\}$ &
$\begin{pspicture}[shift=-1](0,-1.1)(0,1.3) \end{pspicture}
\begin{pspicture}[shift=-1](0,1)(0,1)\begin{pspicture}[shift=-1](-1.4,-1.1)(1.4,1.3)
\pcline(1;18)(1;306)\middlearrow \pcline(1;306)(1;234)\middlearrow
\pcline(1;234)(1;162)\middlearrow \pcline(1;162)(1;90)\middlearrow
\pcline(1;90)(1;18)\middlearrow
\pscircle[fillcolor=lightgray, fillstyle=solid, linewidth=1pt](1;18){.2}
\pscircle[fillcolor=lightgray, fillstyle=solid, linewidth=1pt](1;90){.2}
\pscircle[fillcolor=lightgray, fillstyle=solid, linewidth=1pt](1;162){.2}
\pscircle[fillcolor=lightgray, fillstyle=solid, linewidth=1pt](1;234){.2}
\pscircle[fillcolor=lightgray, fillstyle=solid, linewidth=1pt](1;306){.2}
\rput(1;18){{\footnotesize{0}}}\rput(1;90){{{\footnotesize{1}}}}\rput(1;162){{{\footnotesize{2}}}}
\rput(1;234){{\footnotesize{3}}}\rput(1;306){{\footnotesize{4}}}
\end{pspicture}\end{pspicture}$ & $41$ & $\{0,1,2,4\}$  &
$\begin{pspicture}[shift=-1](0,-1.1)(0,1.3) \end{pspicture}
\begin{pspicture}[shift=-1](0,1)(0,1)\begin{pspicture}[shift=-1](-1.4,-1.1)(1.4,1.3)
\pccurve[angleA=129,angleB=-21,ncurv=1](1;18)(1;90)\middlearrow
\pccurve[angleA=201,angleB=51,ncurv=1](1;90)(1;162)\middlearrow
\pccurve[angleA=273,angleB=123,ncurv=1](1;162)(1;234)\middlearrow
\pccurve[angleA=345,angleB=195,ncurv=1](1;234)(1;306)\middlearrow
\pccurve[angleA=57,angleB=267,ncurv=1](1;306)(1;18)\middlearrow
\pccurve[angleA=-123,angleB=87,ncurv=1](1;18)(1;306)\middlearrow
\pccurve[angleA=-51,angleB=159,ncurv=1](1;90)(1;18)\middlearrow
\pccurve[angleA=21,angleB=231,ncurv=1](1;162)(1;90)\middlearrow
\pccurve[angleA=93,angleB=303,ncurv=1](1;234)(1;162)\middlearrow
\pccurve[angleA=165,angleB=15,ncurv=1](1;306)(1;234)\middlearrow
\pcline(1;18)(1;162)\middlearrow \pcline(1;162)(1;306)\middlearrow
\pcline(1;306)(1;90)\middlearrow \pcline(1;90)(1;234)\middlearrow
\pcline(1;234)(1;18)\middlearrow
\pscircle[fillcolor=lightgray, fillstyle=solid, linewidth=1pt](1;18){.2}
\pscircle[fillcolor=lightgray, fillstyle=solid, linewidth=1pt](1;90){.2}
\pscircle[fillcolor=lightgray, fillstyle=solid, linewidth=1pt](1;162){.2}
\pscircle[fillcolor=lightgray, fillstyle=solid, linewidth=1pt](1;234){.2}
\pscircle[fillcolor=lightgray, fillstyle=solid, linewidth=1pt](1;306){.2}
\rput(1;18){{\footnotesize{0}}}\rput(1;90){{{\footnotesize{1}}}}\rput(1;162){{{\footnotesize{2}}}}
\rput(1;234){{\footnotesize{3}}}\rput(1;306){{\footnotesize{4}}}
\end{pspicture}\end{pspicture}$  & $1121$ \\ \hline
$\{0,1,2\}$ &
$\begin{pspicture}[shift=-1](0,-1.1)(0,1.3) \end{pspicture}
\begin{pspicture}[shift=-1](0,1)(0,1)\begin{pspicture}[shift=-1](-1.4,-1.1)(1.4,1.3)
\pcline(1;18)(1;90)\middlearrow \pcline(1;90)(1;162)\middlearrow
\pcline(1;162)(1;234)\middlearrow \pcline(1;234)(1;306)\middlearrow
\pcline(1;306)(1;18)\middlearrow
\pcline(1;18)(1;162)\middlearrow \pcline(1;162)(1;306)\middlearrow
\pcline(1;306)(1;90)\middlearrow \pcline(1;90)(1;234)\middlearrow
\pcline(1;234)(1;18)\middlearrow
\pscircle[fillcolor=lightgray, fillstyle=solid, linewidth=1pt](1;18){.2}
\pscircle[fillcolor=lightgray, fillstyle=solid, linewidth=1pt](1;90){.2}
\pscircle[fillcolor=lightgray, fillstyle=solid, linewidth=1pt](1;162){.2}
\pscircle[fillcolor=lightgray, fillstyle=solid, linewidth=1pt](1;234){.2}
\pscircle[fillcolor=lightgray, fillstyle=solid, linewidth=1pt](1;306){.2}
\rput(1;18){{\footnotesize{0}}}\rput(1;90){{{\footnotesize{1}}}}\rput(1;162){{{\footnotesize{2}}}}
\rput(1;234){{\footnotesize{3}}}\rput(1;306){{\footnotesize{4}}}
\end{pspicture}\end{pspicture}$ & $113$ & $\{0,1,3,4\}$  &
$\begin{pspicture}[shift=-1](0,-1.1)(0,1.3) \end{pspicture}
\begin{pspicture}[shift=-1](0,1)(0,1)\begin{pspicture}[shift=-1](-1.4,-1.1)(1.4,1.3)
\pccurve[angleA=129,angleB=-21,ncurv=1](1;18)(1;90)\middlearrow
\pccurve[angleA=201,angleB=51,ncurv=1](1;90)(1;162)\middlearrow
\pccurve[angleA=273,angleB=123,ncurv=1](1;162)(1;234)\middlearrow
\pccurve[angleA=345,angleB=195,ncurv=1](1;234)(1;306)\middlearrow
\pccurve[angleA=57,angleB=267,ncurv=1](1;306)(1;18)\middlearrow
\pccurve[angleA=-123,angleB=87,ncurv=1](1;18)(1;306)\middlearrow
\pccurve[angleA=-51,angleB=159,ncurv=1](1;90)(1;18)\middlearrow
\pccurve[angleA=21,angleB=231,ncurv=1](1;162)(1;90)\middlearrow
\pccurve[angleA=93,angleB=303,ncurv=1](1;234)(1;162)\middlearrow
\pccurve[angleA=165,angleB=15,ncurv=1](1;306)(1;234)\middlearrow
\pcline(1;18)(1;234)\middlearrow \pcline(1;234)(1;90)\middlearrow
\pcline(1;90)(1;306)\middlearrow \pcline(1;306)(1;162)\middlearrow
\pcline(1;162)(1;18)\middlearrow
\pscircle[fillcolor=lightgray, fillstyle=solid, linewidth=1pt](1;18){.2}
\pscircle[fillcolor=lightgray, fillstyle=solid, linewidth=1pt](1;90){.2}
\pscircle[fillcolor=lightgray, fillstyle=solid, linewidth=1pt](1;162){.2}
\pscircle[fillcolor=lightgray, fillstyle=solid, linewidth=1pt](1;234){.2}
\pscircle[fillcolor=lightgray, fillstyle=solid, linewidth=1pt](1;306){.2}
\rput(1;18){{\footnotesize{0}}}\rput(1;90){{{\footnotesize{1}}}}\rput(1;162){{{\footnotesize{2}}}}
\rput(1;234){{\footnotesize{3}}}\rput(1;306){{\footnotesize{4}}}
\end{pspicture}\end{pspicture}$ & $1211$ \\ \hline
$\{0,1,3\}$ &
$\begin{pspicture}[shift=-1](0,-1.1)(0,1.3) \end{pspicture}
\begin{pspicture}[shift=-1](0,1)(0,1)\begin{pspicture}[shift=-1](-1.4,-1.1)(1.4,1.3)
\pcline(1;18)(1;90)\middlearrow \pcline(1;90)(1;162)\middlearrow
\pcline(1;162)(1;234)\middlearrow \pcline(1;234)(1;306)\middlearrow
\pcline(1;306)(1;18)\middlearrow
\pcline(1;18)(1;234)\middlearrow \pcline(1;234)(1;90)\middlearrow
\pcline(1;90)(1;306)\middlearrow \pcline(1;306)(1;162)\middlearrow
\pcline(1;162)(1;18)\middlearrow
\pscircle[fillcolor=lightgray, fillstyle=solid, linewidth=1pt](1;18){.2}
\pscircle[fillcolor=lightgray, fillstyle=solid, linewidth=1pt](1;90){.2}
\pscircle[fillcolor=lightgray, fillstyle=solid, linewidth=1pt](1;162){.2}
\pscircle[fillcolor=lightgray, fillstyle=solid, linewidth=1pt](1;234){.2}
\pscircle[fillcolor=lightgray, fillstyle=solid, linewidth=1pt](1;306){.2}
\rput(1;18){{\footnotesize{0}}}\rput(1;90){{{\footnotesize{1}}}}\rput(1;162){{{\footnotesize{2}}}}
\rput(1;234){{\footnotesize{3}}}\rput(1;306){{\footnotesize{4}}}
\end{pspicture}\end{pspicture}$ & $122$ & $\{0,2,3,4\}$  &
$\begin{pspicture}[shift=-1](0,-1.1)(0,1.3) \end{pspicture}
\begin{pspicture}[shift=-1](0,1)(0,1)\begin{pspicture}[shift=-1](-1.4,-1.1)(1.4,1.3)
\pcline(1;18)(1;306)\middlearrow \pcline(1;306)(1;234)\middlearrow
\pcline(1;234)(1;162)\middlearrow \pcline(1;162)(1;90)\middlearrow
\pcline(1;90)(1;18)\middlearrow
\pccurve[angleA=160,angleB=20,ncurv=1](1;18)(1;162)\middlearrow
\pccurve[angleA=232,angleB=92,ncurv=1](1;90)(1;234)\middlearrow
\pccurve[angleA=304,angleB=164,ncurv=1](1;162)(1;306)\middlearrow
\pccurve[angleA=26,angleB=246,ncurv=1](1;234)(1;18)\middlearrow
\pccurve[angleA=88,angleB=308,ncurv=1](1;306)(1;90)\middlearrow
\pcline(1;18)(1;234)\middlearrow \pcline(1;234)(1;90)\middlearrow
\pcline(1;90)(1;306)\middlearrow \pcline(1;306)(1;162)\middlearrow
\pcline(1;162)(1;18)\middlearrow
\pscircle[fillcolor=lightgray, fillstyle=solid, linewidth=1pt](1;18){.2}
\pscircle[fillcolor=lightgray, fillstyle=solid, linewidth=1pt](1;90){.2}
\pscircle[fillcolor=lightgray, fillstyle=solid, linewidth=1pt](1;162){.2}
\pscircle[fillcolor=lightgray, fillstyle=solid, linewidth=1pt](1;234){.2}
\pscircle[fillcolor=lightgray, fillstyle=solid, linewidth=1pt](1;306){.2}
\rput(1;18){{\footnotesize{0}}}\rput(1;90){{{\footnotesize{1}}}}\rput(1;162){{{\footnotesize{2}}}}
\rput(1;234){{\footnotesize{3}}}\rput(1;306){{\footnotesize{4}}}
\end{pspicture}\end{pspicture}$ & $2111$ \\ \hline
$\{0,1,4\}$ &
$\begin{pspicture}[shift=-1](0,-1.1)(0,1.3) \end{pspicture}
\begin{pspicture}[shift=-1](0,1)(0,1)\begin{pspicture}[shift=-1](-1.4,-1.1)(1.4,1.3)
\pccurve[angleA=129,angleB=-21,ncurv=1](1;18)(1;90)\middlearrow
\pccurve[angleA=201,angleB=51,ncurv=1](1;90)(1;162)\middlearrow
\pccurve[angleA=273,angleB=123,ncurv=1](1;162)(1;234)\middlearrow
\pccurve[angleA=345,angleB=195,ncurv=1](1;234)(1;306)\middlearrow
\pccurve[angleA=57,angleB=267,ncurv=1](1;306)(1;18)\middlearrow
\pccurve[angleA=-123,angleB=87,ncurv=1](1;18)(1;306)\middlearrow
\pccurve[angleA=-51,angleB=159,ncurv=1](1;90)(1;18)\middlearrow
\pccurve[angleA=21,angleB=231,ncurv=1](1;162)(1;90)\middlearrow
\pccurve[angleA=93,angleB=303,ncurv=1](1;234)(1;162)\middlearrow
\pccurve[angleA=165,angleB=15,ncurv=1](1;306)(1;234)\middlearrow
\pscircle[fillcolor=lightgray, fillstyle=solid, linewidth=1pt](1;18){.2}
\pscircle[fillcolor=lightgray, fillstyle=solid, linewidth=1pt](1;90){.2}
\pscircle[fillcolor=lightgray, fillstyle=solid, linewidth=1pt](1;162){.2}
\pscircle[fillcolor=lightgray, fillstyle=solid, linewidth=1pt](1;234){.2}
\pscircle[fillcolor=lightgray, fillstyle=solid, linewidth=1pt](1;306){.2}
\rput(1;18){{\footnotesize{0}}}\rput(1;90){{{\footnotesize{1}}}}\rput(1;162){{{\footnotesize{2}}}}
\rput(1;234){{\footnotesize{3}}}\rput(1;306){{\footnotesize{4}}}
\end{pspicture}\end{pspicture}$ & $131$ & $\{0,1,2,3,4\}$  &
$\begin{pspicture}[shift=-1](0,-1.1)(0,1.3) \end{pspicture}
\begin{pspicture}[shift=-1](0,1)(0,1)\begin{pspicture}[shift=-1](-1.4,-1.1)(1.4,1.3)
\pccurve[angleA=124,angleB=-16,ncurv=1](1;18)(1;90)\middlearrow
\pccurve[angleA=196,angleB=56,ncurv=1](1;90)(1;162)\middlearrow
\pccurve[angleA=268,angleB=128,ncurv=1](1;162)(1;234)\middlearrow
\pccurve[angleA=340,angleB=200,ncurv=1](1;234)(1;306)\middlearrow
\pccurve[angleA=52,angleB=272,ncurv=1](1;306)(1;18)\middlearrow
\pcline(1;18)(1;306)\middlearrow \pcline(1;306)(1;234)\middlearrow
\pcline(1;234)(1;162)\middlearrow \pcline(1;162)(1;90)\middlearrow
\pcline(1;90)(1;18)\middlearrow
\pccurve[angleA=165,angleB=15,ncurv=1](1;18)(1;162)\middlearrow
\pccurve[angleA=237,angleB=87,ncurv=1](1;90)(1;234)\middlearrow
\pccurve[angleA=309,angleB=159,ncurv=1](1;162)(1;306)\middlearrow
\pccurve[angleA=31,angleB=241,ncurv=1](1;234)(1;18)\middlearrow
\pccurve[angleA=93,angleB=303,ncurv=1](1;306)(1;90)\middlearrow
\pcline(1;18)(1;234)\middlearrow \pcline(1;234)(1;90)\middlearrow
\pcline(1;90)(1;306)\middlearrow \pcline(1;306)(1;162)\middlearrow
\pcline(1;162)(1;18)\middlearrow
\pscircle[fillcolor=lightgray, fillstyle=solid, linewidth=1pt](1;18){.2}
\pscircle[fillcolor=lightgray, fillstyle=solid, linewidth=1pt](1;90){.2}
\pscircle[fillcolor=lightgray, fillstyle=solid, linewidth=1pt](1;162){.2}
\pscircle[fillcolor=lightgray, fillstyle=solid, linewidth=1pt](1;234){.2}
\pscircle[fillcolor=lightgray, fillstyle=solid, linewidth=1pt](1;306){.2}
\rput(1;18){{\footnotesize{0}}}\rput(1;90){{{\footnotesize{1}}}}\rput(1;162){{{\footnotesize{2}}}}
\rput(1;234){{\footnotesize{3}}}\rput(1;306){{\footnotesize{4}}}
\end{pspicture}\end{pspicture}$  & $11111$ \\ \hline
\end{tabular}
\vskip .2cm
\caption{$\Omega \in \mathcal{G}_C(5)$, $G(5,\Omega)$ and $\sigma_{\Omega}\in C(5)$} \label{com5exafig}
\end{table}

In the following, we define two notations which will be used to consider the relationship between $C(n)$ and $\mathcal{G}_C(n)$. Let $C_n$ be the cycle of length $n$ with vertex set $\B{Z}_n=\{0, 1,2, \ldots, n-1\}$ (see Figure \ref{figsigmaS} (a)). We consider each element $\Omega \in \mathcal{G}_C(n)$ as a subset of the vertex set of $C_n$. For each element $\Omega=\{a_1,a_2\ldots,a_t\}$ of $\mathcal{G}_C(n)$, we define a composition
\begin{equation}\label{sigmaS}
\sigma_{\Omega}=\omega_1\omega_2\cdots \omega_t \in C(n)
\end{equation}
where $\omega_i= a_{i+1}-a_i$ for each $i=1, 2, \ldots, t-1$ and $\omega_t=n-a_t$ (see Figure \ref{figsigmaS} (b)). Notice that $\omega_i$ is the number of encountered edges when we move anticlockwise from $a_i$ to $a_{i+1}$ for each $1\leq i\leq t$ where $a_{t+1}:=a_1$.

\begin{figure}
$$
\begin{pspicture}[shift=-2.4](-2.5,-3.1)(2.5,2.5)
\psarc(0,0){2}{-70}{250}
\psarc[linestyle=dotted, linewidth=1.2pt](0,0){2}{250}{-70}
\pscircle[fillcolor=lightgray, fillstyle=solid, linewidth=1pt](2;-60){.1}
\pscircle[fillcolor=lightgray, fillstyle=solid, linewidth=1pt](2;-40){.1}
\pscircle[fillcolor=lightgray, fillstyle=solid, linewidth=1pt](2;-20){.1}
\pscircle[fillcolor=lightgray, fillstyle=solid, linewidth=1pt](2;0){.1}
\pscircle[fillcolor=lightgray, fillstyle=solid, linewidth=1pt](2;20){.1}
\pscircle[fillcolor=lightgray, fillstyle=solid, linewidth=1pt](2;40){.1}
\pscircle[fillcolor=lightgray, fillstyle=solid, linewidth=1pt](2;60){.1}
\pscircle[fillcolor=lightgray, fillstyle=solid, linewidth=1pt](2;80){.1}
\pscircle[fillcolor=lightgray, fillstyle=solid, linewidth=1pt](2;100){.1}
\pscircle[fillcolor=lightgray, fillstyle=solid, linewidth=1pt](2;120){.1}
\pscircle[fillcolor=lightgray, fillstyle=solid, linewidth=1pt](2;140){.1}
\pscircle[fillcolor=lightgray, fillstyle=solid, linewidth=1pt](2;160){.1}
\pscircle[fillcolor=lightgray, fillstyle=solid, linewidth=1pt](2;180){.1}
\pscircle[fillcolor=lightgray, fillstyle=solid, linewidth=1pt](2;200){.1}
\pscircle[fillcolor=lightgray, fillstyle=solid, linewidth=1pt](2;220){.1}
\pscircle[fillcolor=lightgray, fillstyle=solid, linewidth=1pt](2;240){.1}
\pscircle[fillcolor=white, linecolor=white, fillstyle=solid](2.5;-22.5){.32}
\pscircle[fillcolor=white, linecolor=white, fillstyle=solid](2.5;30){.32}
\pscircle[fillcolor=white, linecolor=white, fillstyle=solid](2.5;105){.32}
\rput(1.45;-20){$n-1$} \rput(1.7;0){$0$} \rput(1.7;20){$1$}  \rput(1.7;40){$2$} \rput(1.7;60){$3$}
\rput(1.7;80){$4$} \rput(1.7;100){$5$}
\rput[t](0,-2.5){$(a)$}
\end{pspicture}\quad\quad
\begin{pspicture}[shift=-2.4](-2.5,-3.1)(3.2,2.5)
\psarc(0,0){2}{-70}{250}
\psarc[linestyle=dotted, linewidth=1.2pt](0,0){2}{250}{-70}
\pccurve[angleA=20,angleB=-45, linestyle=dashed](2.5;-70)(2;-45)
\pccurve[angleA=-45,angleB=-112.5, linestyle=dashed](2;-45)(2.5;-22.5)
\pccurve[angleA=67.5,angleB=0, linestyle=dashed](2.5;-22.5)(2;0)
\pccurve[angleA=0,angleB=-60, linestyle=dashed](2;0)(2.5;30)
\pccurve[angleA=120,angleB=60, linestyle=dashed](2.5;30)(2;60)
\pccurve[angleA=60,angleB=15, linestyle=dashed](2;60)(2.5;105)
\pccurve[angleA=-165,angleB=150, linestyle=dashed](2.5;105)(2;150)
\pccurve[angleA=150,angleB=90, linestyle=dashed](2;150)(2.5;180)
\pscircle[fillcolor=lightgray, fillstyle=solid, linewidth=1pt](2;-45){.1}
\pscircle[fillcolor=lightgray, fillstyle=solid, linewidth=1pt](2;0){.1}
\pscircle[fillcolor=lightgray, fillstyle=solid, linewidth=1pt](2;60){.1}
\pscircle[fillcolor=lightgray, fillstyle=solid, linewidth=1pt](2;150){.1}
\pscircle[fillcolor=white, linecolor=white, fillstyle=solid](2.5;-22.5){.32}
\pscircle[fillcolor=white, linecolor=white, fillstyle=solid](2.5;30){.32}
\pscircle[fillcolor=white, linecolor=white, fillstyle=solid](2.5;105){.32}
\rput(3;0){$a_1=0$} \rput(1.7;-45){$a_t$} \rput(1.7;60){$a_2$}  \rput(1.7;150){$a_3$} 
\rput(2.5;-22.5){$\omega_t$} \rput(2.5;30){$\omega_1$} \rput(2.5;105){$\omega_2$}
\rput[t](0,-2.5){$(b)$}
\end{pspicture}
$$
\caption{$(a)$ $C_{n}$      $(b)$ $\Omega=\{ a_1$, $a_2$, $\ldots$, $a_t\}$ and $\sigma_{\Omega}
=\omega_1\omega_2 \cdots \omega_t$} \label{figsigmaS}
\end{figure}
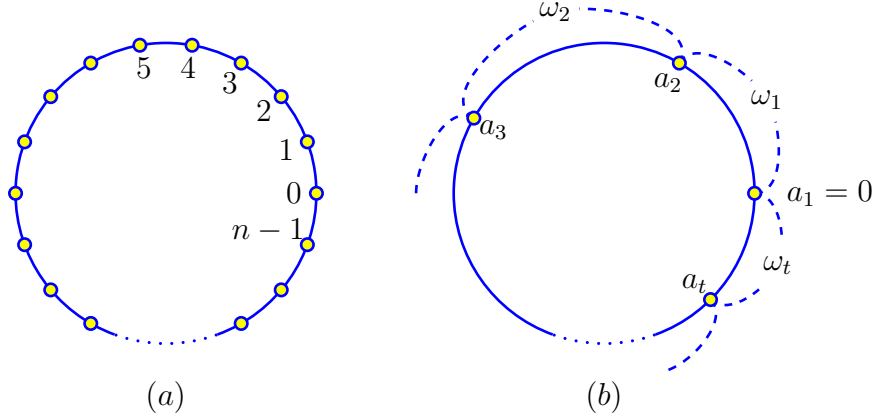

Conversely, for each composition $\sigma=\sigma_1 \sigma_2\cdots \sigma_{\ell}\in C(n)$, we define a set
\begin{equation}\label{Ssigma}
\Omega_{\sigma}=\{s_1,s_2\cdots s_{\ell}\}\in \mathcal{G}_C(n)
\end{equation}
where $s_1=0$, $s_i=\sum_{j=1}^{i-1}\sigma_j $ for each $i=2, \ldots, \ell $. Note that if $\sigma=n\in C(n)$ and $\sigma '=\overbrace{11\cdots 1}^\text{$n$ times of $1$} \in C(n)$ then $\Omega_{\sigma}=\{0\}$ and $\Omega_{\sigma '}=\{0,1,2,\ldots,n-1\}$, respectively.\\

In the following theorem, we will show that these constructions $\sigma_{\Omega}$ and $\Omega_{\sigma}$ in (\ref{sigmaS}) and (\ref{Ssigma}) give a one-to-one correspondence between the sets $\mathcal{G}_C(n)$ and $C(n)$.

\begin{thm}\label{dg&c}
For each integer $n\geq 1$, the map $\psi : \mathcal{G}_C(n) \longrightarrow C(n)$ defined by $\psi (\Omega)= \sigma_{\Omega}$ is a bijection (i.e., there is a one-to-one correspondence between the set of circulant digraphs of order $n$ and the set of compositions of $n$.).
\end{thm}
\begin{proof}
It follows by (\ref{sigmaS}) that the map $\psi$ is well-defined and injective. For any composition $\sigma \in C(n)$, a set $\Omega_{\sigma}$ in $\mathcal{G}_C(n)$ satisfies $\psi(\Omega_{\sigma})=\sigma$, see (\ref{Ssigma}). Thus the map $\psi$ is surjective and this completes the proof.
\end{proof}

Let $n\geq 2$ be an integer. For each integer $1\leq k\leq n$, let
$$ C(n,k) = \{ \sigma\in C(n) \mid \sigma \mbox{~has~$k$~parts}\}\mbox{~and~} \mathcal{G}_{C}(n,k) = \{ \Omega \in  \mathcal{G}_C(n)\mid  |\Omega|=k \}.$$
For each $\Omega\in \mathcal{G}_C(n,k)$, the corresponding circulant digraph $G(\BZ_n,\Omega)$ is a regular graph with outdegree $k-1$. Since the bijection $\psi$ in Theorem \ref{dg&c} naturally matches $\mathcal{G}_C(n,k)$ and $C(n,k)$, and by $\left|\mathcal{G}_C(n,k)\right|=\binom{n-1}{k-1}$, we can find immediately the number of compositions of $n$ with $k$ parts as follows.

\begin{cor} \label{dicirculantvalencyk}
Let $n\geq 2$ be an integer. For each integer $1\leq k\leq n$,
the number of compositions of $n$ with $k$ parts is $\binom{n-1}{k-1}$.
\end{cor}

For each $\sigma=\sigma_1 \sigma_2\cdots \sigma_{\ell}\in C(n)$ and for each $\Omega=\{a_1,a_2\ldots,a_t\}\in \mathcal{G}_C(n)$, define
 $$\gcd(\sigma) = \gcd\{\sigma_1, \sigma_2, \ldots, \sigma_{\ell}\}\mbox{~and~} \gcd(\Omega)=\gcd \{a_1,a_2\ldots,a_t\}. $$
For each integer $n\geq 2$, we define
\[ C(n)^* = \{ \sigma \in C(n) \mid \gcd (\sigma)=1\} \mbox{~and~} \mathcal{G}_C(n)^* = \{ \Omega \in \mathcal{G}_C(n) \mid \gcd (\Omega) =1 \}.\]

In the following theorem, we will show that the bijection $\psi$ in Theorem \ref{dg&c} guarantees a one-to-one correspondence between $\mathcal{G}_C(n)^*$ and $C(n)^*$.
\begin{thm} \label{dicirculantthm}
For each integer $n\geq 2$, there is a one-to-one correspondence between $\mathcal{G}_C(n)^*$ and $ C(n)^*$.
\end{thm}
\begin{proof}
To prove the result, we will show that the map $\psi|_{\mathcal{G}_C(n)^*} : \mathcal{G}_C(n)^* \longrightarrow C(n)^*$ defined by $\psi|_{\mathcal{G}_C(n)^*} (\Omega)=\psi(\Omega)$ is a bijection, where $\psi: \mathcal{G}_C(n) \longrightarrow C(n)$ is the bijection defined in Theorem \ref{dg&c}. It is enough to show that $\psi|_{\mathcal{G}_C(n)^*}$ is well-defined and surjective since it is injective by Theorem \ref{dg&c}. For each $\Omega=\{a_1$, $a_2$, $\ldots$, $a_t \} \in \mathcal{G}_C(n)^*$, $\gcd \{a_1, a_2, \ldots, a_t \}=1$ and thus it follows by the Euclidean algorithm that
\[1=\gcd \{a_1, a_2, \ldots, a_t\} =\gcd \{a_2-a_1, a_3-a_2, \ldots, n-a_t \}=\gcd (\sigma_{\Omega})=\gcd(\psi|_{\mathcal{G}_C(n)^*}(\Omega)).\]
Hence $\psi|_{\mathcal{G}_C(n)^*}(\Omega) \in C(n)^*$ and thus $\psi|_{\mathcal{G}_C(n)^*}$ is well-defined. Similary, one can check that $\psi|_{\mathcal{G}_C(n)^*}$ is surjective since for each $\sigma \in C(n)^*$, $\psi|_{\mathcal{G}_C(n)^*}(\Omega_{\sigma})=\sigma$ holds for $\Omega_{\sigma}\in \mathcal{G}_C(n)^*$.
\end{proof}

As an application, we now enumerate the number of connected circulant digraphs. It follows by Theorem \ref{dg&c} that the cardinality of $C(n)$ is equal to the number of subsets of $\B{Z}_n$ that contain $0$, i.e., $|C(n)|=2^{n-1}$. By Theorem \ref{dicirculantthm}, $|C(n)^*|$ is equal to the number of generating set of $\B{Z}_n$ that contains $0$. Hence $|C(n)^*|$ is the number of connected circulant digraphs of order $n$, and thus the cardinality of $C(n)\setminus C(n)^* = \{ \sigma \in C(n) \mid \gcd (\sigma)\ne 1\} $ is the number of disconnected circulant digraphs of order $n$.
In the following corollary, we enumerate these two numbers $|C(n)^*|$ and $\left| C(n)\setminus C(n)^* \right|$.

\begin{cor} \label{ancounting}
Let $n\ge 2$ be an integer. Then the following hold:
\begin{enumerate}
\item[(i)] $|C(n)|=2^{n-1}$,
\item[(ii)] $|C(n)^*| = \sum_{d | n} \mu\left(\frac{n}{d}\right) 2^{d-1}$ and
\item[(iii)] $|C(n)\setminus C(n)^*|=\sum_{d|n, d\not= n} |C(d)^*| =\sum_{d | n, d\not=n} -\mu\left(\frac{n}{d}\right) 2^{d-1}$,
\end{enumerate}
where $\mu$ is the M\"obius function.
\end{cor}
\begin{proof}
(i): It follows by Theorem \ref{dg&c}.\\
(ii)-(iii): Note that for each $\sigma=\sigma_1\cdots \sigma_{\ell}\in C(n)$ with $\gcd(\sigma)=d\geq 1$, composition $\frac{1}{d}\sigma:=\left(\frac{\sigma_1}{d}\right)\cdots \left(\frac{\sigma_{\ell}}{d}\right)$ satisfies $\frac{1}{d}\sigma\in C\left(\frac{n}{d} \right)^*$. We now observe $|C(n)|=|\cup_{d|n}C(d)^*|=\sum_{d|n}|C(d)^*|$. Now, by applying the M\"obius inversion formula with $|C(d)|=2^{d-1}$ by (i), the result (ii) follows. Now (iii) follows as $|C(n)\setminus C(n)^*|=|C(n)|-|C(n)^*|$ and $\mu(1)=1$.
\end{proof}

Using Corollary~\ref{ancounting}, we enumerate $|C(n)^*|$ and $|C(n)\setminus C(n)^*|$ up to $n=40$ in Table~\ref{abntable}.

\begin{table}
\begin{tabular}{||c|c|c|c|c|c|c|c|c|c|c||}\hline\hline
$n$         & 1 & 2 & 3 & 4 & 5 & 6 & 7 & 8 & 9 & 10 \\ \hline
$|C(n)^*|$ & \footnotesize{1} & \footnotesize{1} & \footnotesize{3} & \footnotesize{6}
&\footnotesize{15} &\footnotesize{27} &\footnotesize{63} &\footnotesize{120}&\footnotesize{252}&\footnotesize{495} \\ \hline
$|C(n)\setminus C(n)^*|$  & \footnotesize{0} & \footnotesize{1} & \footnotesize{1} & \footnotesize{2} & \footnotesize{1}
& \footnotesize{5} & \footnotesize{1} & \footnotesize{8} & \footnotesize{4} & \footnotesize{17} \\ \hline\hline
$n$         & 11 & 12 & 13 & 14 & 15 & 16 & 17 & 18 & 19 & 20 \\ \hline
$|C(n)^*|$ &\footnotesize{1,023}&\footnotesize{2,010}&4,095&8,127&\footnotesize{16,365}&\footnotesize{32,640}
&\footnotesize{65,535}&\footnotesize{130,788}&\footnotesize{262,143}&\footnotesize{523,770} \\ \hline
$|C(n)\setminus C(n)^*|$  & \footnotesize{1}  & \footnotesize{38} & \footnotesize{1}  & \footnotesize{65}
& \footnotesize{19} & \footnotesize{128} & \footnotesize{1} & \footnotesize{284} & \footnotesize{1} & \footnotesize{518} \\ \hline\hline
$n$         & 21 & 22 & 23 & 24 & 25 & 26 & 27 & 28 & 29 & 30 \\ \hline
$|C(n)\setminus C(n)^*|$  & \footnotesize{67} & \footnotesize{1,025} & \footnotesize{1} & \footnotesize{2,168}
& \footnotesize{16} & \footnotesize{4,097} & \footnotesize{256} & \footnotesize{9,198} & \footnotesize{1} & \footnotesize{16,905} \\ \hline\hline
$n$         & 31 & 32 & 33 & 34 & 35 & 36 & 37 & 38 & 39 & 40 \\ \hline
$|C(n)\setminus C(n)^*|$  & \footnotesize{1}&\footnotesize{32,768}&\footnotesize{1,027}& \footnotesize{65,537} & \footnotesize{79} & \footnotesize{133,090}
& \footnotesize{1} & \footnotesize{262,145} & \footnotesize{4,099} & \footnotesize{524,282} \\ \hline
\hline
\end{tabular}
\vskip .2cm
\caption{$|C(n)^*|$ and $|C(n)\setminus C(n)^*|$ up to $n=40$.} \label{abntable}
\end{table}

\begin{exa} \label{enumeration-exa}
The number of connected circulant digraphs
and the number of disconnected circulant digraphs of order $72$ are $23,611,832,414,004,545,432,040$
and $34,368,074,808$, respectively.
\begin{proof}
The prime factorization of $72$ is $2^3 \cdot 3^2$. There are
$11$ proper divisors, $1$, $2$, $3$, $4$, $6$, $8$, $9$, $12$,
$18$, $24$ and $36$. In Table~\ref{abntable}, we find their $|C(k)^*|$
as $1$, $1$, $3$, $6$, $25$, $120$, $252$, $2012$, $130790$, $8336440$ and
$34359605278$. By summing all, we have $|C(72)\setminus C(72)^*|=34,368,074,808$ and
$|C(72)^*|=23,611,832,414,004,545,432,040$.
\end{proof}
\end{exa}

\section{One-to-one correspondence between the palindromes of $n$ and the circulant graphs of order $n$}\label{undirected}
In this section, we first show that there is a bijection between the set of palindromes of $n$ and the set of circulant graphs of order $n$ in Theorem \ref{palthm}. In particular, we also show that this bijection in Theorem \ref{palthm} guarantees a one-to-one correspondence between the set of aperiodic palindromes of $n$ and the set of connected circulant graphs of order $n$ (see Theorem \ref{apn-thm}). \\

For each integer $n\geq 2$, we define
\begin{equation*}
\mathcal{G}_P(n) = \{ \Omega \in \mathcal{G}_C(n) \mid \Omega=\Omega^{-1} \} \mbox{~and~}
P(n) = \{ \sigma\in C(n) \mid \sigma=\sigma^{-1}\}.
\end{equation*}
Note that the set $P(n)$ is the set of palindromes of $n$. Without loss of generality, each element $\Omega \in \mathcal{G}_P(n)$ with $|\Omega|=t\geq 2$ is denoted
 by a set $\Omega=\{ a_1, a_2, \ldots, a_t\}$ satisfying
 \begin{equation}\label{gpn}
 0=a_1<a_2<\cdots<a_t \mbox{~and~} a_i+a_{t+2-i}=n,~~i=2,\ldots, t.
\end{equation}

In the following theorem, we will show that the construction $\sigma_{\Omega}$ in (\ref{sigmaS})
for an element $\Omega$ of $\mathcal{G}_P(n)$ guarantees one-to-one correspondence between $\mathcal{G}_P(n)$ and $P(n)$.

\begin{thm} \label{palthm}
For each integer $n\geq 2$, there is a one-to-one correspondence between the set of palindromes of $n$ and the set of circulant graphs of order $n$.
\end{thm}
\begin{proof}
To prove the result, we will show that the map $\psi|_{\mathcal{G}_P(n)}: \mathcal{G}_P(n) \longrightarrow P(n)$ defined by $\psi|_{\mathcal{G}_P(n)}(\Omega)=\psi(\Omega)=\sigma_{\Omega}$ is a bijection, where $\psi : \mathcal{G}_C(n) \longrightarrow C(n)$ is the bijection defined in Theorem \ref{dg&c}. To show that the map $\psi|_{\mathcal{G}_P(n)}$ is well-defined and injective, it is enough to show $\sigma_{\Omega}=\sigma_{\Omega}^{-1}$ for each $\Omega=\{ a_1, \ldots, a_t\} \in \mathcal{G}_P(n)$ as $\psi$ is bijective. If $t=1$ then $\Omega=\{a_1=0\}$ and $\sigma_{\Omega}=n=\sigma_{\Omega}^{-1}$. Now let $t\geq 2$ and $\sigma_{\Omega}=\omega_1\cdots \omega_{t}$ in (\ref{sigmaS}). As $\omega_{t}=n-a_t=a_2=\omega_1$ and $\omega_{t+1-i}=a_{t+2-i}-a_{t+1-i}=(n-a_{i})-(n-a_{i+1})=a_{i+1}-a_{i}=\omega_i ~~(i=2,\ldots,t-1)$ follow by (\ref{sigmaS}) and (\ref{gpn}), we find $\sigma_{\Omega}=\sigma_{\Omega}^{-1}$. Let  $\sigma=\sigma_1\cdots \sigma_m$ be an element in $P(n)$ and let $\Omega_{\sigma}=\{a_1,\ldots,a_{m} \}$ (see (\ref{Ssigma})). To show $\psi|_{\mathcal{G}_P(n)}$ is surjective, we will show $\Omega_{\sigma}=\Omega_{\sigma}^{-1}$ as $\Omega_{\sigma}\in \mathcal{G}_C(n)$ and $\psi(\Omega_{\sigma})=\sigma$ in Theorem \ref{dg&c}. If $m=1$ then $\sigma=n$ and thus we find $\Omega_{\sigma}=\{0\}=\Omega_{\sigma}^{-1}$. Let $m\geq 2$. As $\sigma \in P(n)$, $\sigma_i=\sigma_{m+1-i}$
for all $i=1,\ldots,m$ and $\sum_{j=1}^{m}\sigma_j=n$. Thus $a_1=0$, $a_2+a_{m}=\sigma_1+\sum_{j=1}^{m-1}\sigma_j=\sum_{j=1}^{m}\sigma_j=n$ by (\ref{Ssigma}), and 
$$a_i+a_{m+2-i}
=\sum_{j=1}^{i-1}\sigma_j +\sum_{k=1}^{m+1-i}\sigma_k
=\sum_{j=1}^{i-1}\sigma_j +\sum_{k=1}^{m+1-i}\sigma_{m+1-k}
=\sum_{j=1}^{m}\sigma_j=n,~~i=2,\ldots,m.
$$
This shows $\Omega_{\sigma}=\Omega_{\sigma}^{-1}$, and thus the map $\psi|_{\mathcal{G}_P(n)}$ is surjective. This completes the proof.
\end{proof}

As we mentioned earlier, it is already known that the number of palindromes of $n\geq 2$
is $2^{\lfloor\frac{n}{2} \rfloor}$ (see \cite{HCG}). By using Theorem \ref{palthm}, we have very short and elementary proof for the same statement.

\begin{cor} \label{counting}
For each integer $n\geq 2$, the number of palindromes of $n$ is $2^{\lfloor\frac{n}{2} \rfloor}$.
\end{cor}
\begin{proof}
Let $n\geq 2$ be an integer and consider $\lfloor\frac{n}{2} \rfloor$ sets $\{i, n-i\}$, $i=1,\ldots, \lfloor\frac{n}{2} \rfloor $. As $\Omega^{-1}=\{n-i \mid i\in \Omega\}$,  the cardinality of the set $\mathcal{G}_P(n) = \{ \Omega \in \mathcal{G}_C(n) \mid \Omega=\Omega^{-1} \}$ is $2^{\lfloor\frac{n}{2} \rfloor}$. As the number of palindromes of $n$ is equal to $|\mathcal{G}_P(n)|$ by Theorem~\ref{palthm}, the result follows.
\end{proof}

\begin{exa} \label{aperiodic8exa}
In Table~\ref{pal8exafig}, the circulant graphs of order $8$ and the corresponding palindromes in Theorem \ref{palthm} are presented.
\end{exa}

\begin{table}
\begin{tabular}{||c|c|c||c|c|c||}\hline
$\Omega \in \mathcal{G}_P(8)$ & $G(8, \Omega)$ & $\sigma_{\Omega}\in P(8)$ & $\Omega \in \mathcal{G}_P(8)$  & $G(8, \Omega)$ &  $\sigma_{\Omega}\in P(8)$ \\
\hline$\{0\}$ & $\begin{pspicture}[shift=-1](-1.4,-1.3)(1.4,1.3)
\pscircle[fillcolor=lightgray, fillstyle=solid, linewidth=1pt](1;22.5){.2}
\pscircle[fillcolor=lightgray, fillstyle=solid, linewidth=1pt](1;67.5){.2}
\pscircle[fillcolor=lightgray, fillstyle=solid, linewidth=1pt](1;112.5){.2}
\pscircle[fillcolor=lightgray, fillstyle=solid, linewidth=1pt](1;157.5){.2}
\pscircle[fillcolor=lightgray, fillstyle=solid, linewidth=1pt](1;202.5){.2}
\pscircle[fillcolor=lightgray, fillstyle=solid, linewidth=1pt](1;247.5){.2}
\pscircle[fillcolor=lightgray, fillstyle=solid, linewidth=1pt](1;292.5){.2}
\pscircle[fillcolor=lightgray, fillstyle=solid, linewidth=1pt](1;337.5){.2}
\rput(1;22.5){{\footnotesize{7}}}\rput(1;67.5){{{\footnotesize{6}}}}\rput(1;112.5){{{\footnotesize{5}}}}
\rput(1;157.5){{\footnotesize{4}}}\rput(1;202.5){{\footnotesize{3}}}\rput(1;247.5){{\footnotesize{2}}}
\rput(1;292.5){{\footnotesize{1}}}\rput(1;337.5){{\footnotesize{0}}}
\end{pspicture}$ &  $8$  &
 $\{0,1,2,6,7\}$  &$\begin{pspicture}[shift=-1](-1.4,-1.3)(1.4,1.3)
\psline(1;22.5)(1;67.5)(1;112.5)(1;157.5)(1;202.5)(1;247.5)(1;292.5)(1;337.5)(1;22.5)
\psline(1;22.5)(1;112.5)(1;202.5)(1;292.5)(1;22.5)
\psline(1;67.5)(1;157.5)(1;247.5)(1;337.5)(1;67.5)
\pscircle[fillcolor=lightgray, fillstyle=solid, linewidth=1pt](1;22.5){.2}
\pscircle[fillcolor=lightgray, fillstyle=solid, linewidth=1pt](1;67.5){.2}
\pscircle[fillcolor=lightgray, fillstyle=solid, linewidth=1pt](1;112.5){.2}
\pscircle[fillcolor=lightgray, fillstyle=solid, linewidth=1pt](1;157.5){.2}
\pscircle[fillcolor=lightgray, fillstyle=solid, linewidth=1pt](1;202.5){.2}
\pscircle[fillcolor=lightgray, fillstyle=solid, linewidth=1pt](1;247.5){.2}
\pscircle[fillcolor=lightgray, fillstyle=solid, linewidth=1pt](1;292.5){.2}
\pscircle[fillcolor=lightgray, fillstyle=solid, linewidth=1pt](1;337.5){.2}
\rput(1;22.5){{\footnotesize{7}}}\rput(1;67.5){{{\footnotesize{6}}}}\rput(1;112.5){{{\footnotesize{5}}}}
\rput(1;157.5){{\footnotesize{4}}}\rput(1;202.5){{\footnotesize{3}}}\rput(1;247.5){{\footnotesize{2}}}
\rput(1;292.5){{\footnotesize{1}}}\rput(1;337.5){{\footnotesize{0}}}
\end{pspicture}$ & $11411$ \\
\hline $\{0,4\}$  & $\begin{pspicture}[shift=-1](-1.4,-1.3)(1.4,1.3)
\psline(1;22.5)(1;202.5)\psline(1;67.5)(1;247.5)\psline(1;112.5)(1;292.5)\psline(1;157.5)(1;337.5)
\pscircle[fillcolor=lightgray, fillstyle=solid, linewidth=1pt](1;22.5){.2}
\pscircle[fillcolor=lightgray, fillstyle=solid, linewidth=1pt](1;67.5){.2}
\pscircle[fillcolor=lightgray, fillstyle=solid, linewidth=1pt](1;112.5){.2}
\pscircle[fillcolor=lightgray, fillstyle=solid, linewidth=1pt](1;157.5){.2}
\pscircle[fillcolor=lightgray, fillstyle=solid, linewidth=1pt](1;202.5){.2}
\pscircle[fillcolor=lightgray, fillstyle=solid, linewidth=1pt](1;247.5){.2}
\pscircle[fillcolor=lightgray, fillstyle=solid, linewidth=1pt](1;292.5){.2}
\pscircle[fillcolor=lightgray, fillstyle=solid, linewidth=1pt](1;337.5){.2}
\rput(1;22.5){{\footnotesize{7}}}\rput(1;67.5){{{\footnotesize{6}}}}\rput(1;112.5){{{\footnotesize{5}}}}
\rput(1;157.5){{\footnotesize{4}}}\rput(1;202.5){{\footnotesize{3}}}\rput(1;247.5){{\footnotesize{2}}}
\rput(1;292.5){{\footnotesize{1}}}\rput(1;337.5){{\footnotesize{0}}}
\end{pspicture}$ &  $(4)^2$   & $\{0,1,3,5,7\}$ &
$\begin{pspicture}[shift=-1](-1.4,-1.3)(1.4,1.3)
\psline(1;22.5)(1;67.5)(1;112.5)(1;157.5)(1;202.5)(1;247.5)(1;292.5)(1;337.5)(1;22.5)
\psline(1;22.5)(1;157.5)(1;292.5)(1;67.5)(1;202.5)(1;337.5)(1;112.5)(1;247.5)(1;22.5)
\pscircle[fillcolor=lightgray, fillstyle=solid, linewidth=1pt](1;22.5){.2}
\pscircle[fillcolor=lightgray, fillstyle=solid, linewidth=1pt](1;67.5){.2}
\pscircle[fillcolor=lightgray, fillstyle=solid, linewidth=1pt](1;112.5){.2}
\pscircle[fillcolor=lightgray, fillstyle=solid, linewidth=1pt](1;157.5){.2}
\pscircle[fillcolor=lightgray, fillstyle=solid, linewidth=1pt](1;202.5){.2}
\pscircle[fillcolor=lightgray, fillstyle=solid, linewidth=1pt](1;247.5){.2}
\pscircle[fillcolor=lightgray, fillstyle=solid, linewidth=1pt](1;292.5){.2}
\pscircle[fillcolor=lightgray, fillstyle=solid, linewidth=1pt](1;337.5){.2}
\rput(1;22.5){{\footnotesize{7}}}\rput(1;67.5){{{\footnotesize{6}}}}\rput(1;112.5){{{\footnotesize{5}}}}
\rput(1;157.5){{\footnotesize{4}}}\rput(1;202.5){{\footnotesize{3}}}\rput(1;247.5){{\footnotesize{2}}}
\rput(1;292.5){{\footnotesize{1}}}\rput(1;337.5){{\footnotesize{0}}}
\end{pspicture}$ &   $12221$ \\
\hline  $\{0,1,7\}$  &$\begin{pspicture}[shift=-1](-1.4,-1.3)(1.4,1.3)
\psline(1;22.5)(1;67.5)(1;112.5)(1;157.5)(1;202.5)(1;247.5)(1;292.5)(1;337.5)(1;22.5)
\pscircle[fillcolor=lightgray, fillstyle=solid, linewidth=1pt](1;22.5){.2}
\pscircle[fillcolor=lightgray, fillstyle=solid, linewidth=1pt](1;67.5){.2}
\pscircle[fillcolor=lightgray, fillstyle=solid, linewidth=1pt](1;112.5){.2}
\pscircle[fillcolor=lightgray, fillstyle=solid, linewidth=1pt](1;157.5){.2}
\pscircle[fillcolor=lightgray, fillstyle=solid, linewidth=1pt](1;202.5){.2}
\pscircle[fillcolor=lightgray, fillstyle=solid, linewidth=1pt](1;247.5){.2}
\pscircle[fillcolor=lightgray, fillstyle=solid, linewidth=1pt](1;292.5){.2}
\pscircle[fillcolor=lightgray, fillstyle=solid, linewidth=1pt](1;337.5){.2}
\rput(1;22.5){{\footnotesize{7}}}\rput(1;67.5){{{\footnotesize{6}}}}\rput(1;112.5){{{\footnotesize{5}}}}
\rput(1;157.5){{\footnotesize{4}}}\rput(1;202.5){{\footnotesize{3}}}\rput(1;247.5){{\footnotesize{2}}}
\rput(1;292.5){{\footnotesize{1}}}\rput(1;337.5){{\footnotesize{0}}}
\end{pspicture}$ &  $161$   &
 $\{0,2,3,5,6\}$  & $\begin{pspicture}[shift=-1](-1.4,-1.3)(1.4,1.3)
\psline(1;22.5)(1;112.5)(1;202.5)(1;292.5)(1;22.5)
\psline(1;67.5)(1;157.5)(1;247.5)(1;337.5)(1;67.5)
\psline(1;22.5)(1;157.5)(1;292.5)(1;67.5)(1;202.5)(1;337.5)(1;112.5)(1;247.5)(1;22.5)
\pscircle[fillcolor=lightgray, fillstyle=solid, linewidth=1pt](1;22.5){.2}
\pscircle[fillcolor=lightgray, fillstyle=solid, linewidth=1pt](1;67.5){.2}
\pscircle[fillcolor=lightgray, fillstyle=solid, linewidth=1pt](1;112.5){.2}
\pscircle[fillcolor=lightgray, fillstyle=solid, linewidth=1pt](1;157.5){.2}
\pscircle[fillcolor=lightgray, fillstyle=solid, linewidth=1pt](1;202.5){.2}
\pscircle[fillcolor=lightgray, fillstyle=solid, linewidth=1pt](1;247.5){.2}
\pscircle[fillcolor=lightgray, fillstyle=solid, linewidth=1pt](1;292.5){.2}
\pscircle[fillcolor=lightgray, fillstyle=solid, linewidth=1pt](1;337.5){.2}
\rput(1;22.5){{\footnotesize{7}}}\rput(1;67.5){{{\footnotesize{6}}}}\rput(1;112.5){{{\footnotesize{5}}}}
\rput(1;157.5){{\footnotesize{4}}}\rput(1;202.5){{\footnotesize{3}}}\rput(1;247.5){{\footnotesize{2}}}
\rput(1;292.5){{\footnotesize{1}}}\rput(1;337.5){{\footnotesize{0}}}
\end{pspicture}$ &  $21212$  \\
\hline   $\{0,2,6\}$   & $\begin{pspicture}[shift=-1](-1.4,-1.3)(1.4,1.3)
\psline(1;22.5)(1;112.5)(1;202.5)(1;292.5)(1;22.5)
\psline(1;67.5)(1;157.5)(1;247.5)(1;337.5)(1;67.5)
\pscircle[fillcolor=lightgray, fillstyle=solid, linewidth=1pt](1;22.5){.2}
\pscircle[fillcolor=lightgray, fillstyle=solid, linewidth=1pt](1;67.5){.2}
\pscircle[fillcolor=lightgray, fillstyle=solid, linewidth=1pt](1;112.5){.2}
\pscircle[fillcolor=lightgray, fillstyle=solid, linewidth=1pt](1;157.5){.2}
\pscircle[fillcolor=lightgray, fillstyle=solid, linewidth=1pt](1;202.5){.2}
\pscircle[fillcolor=lightgray, fillstyle=solid, linewidth=1pt](1;247.5){.2}
\pscircle[fillcolor=lightgray, fillstyle=solid, linewidth=1pt](1;292.5){.2}
\pscircle[fillcolor=lightgray, fillstyle=solid, linewidth=1pt](1;337.5){.2}
\rput(1;22.5){{\footnotesize{7}}}\rput(1;67.5){{{\footnotesize{6}}}}\rput(1;112.5){{{\footnotesize{5}}}}
\rput(1;157.5){{\footnotesize{4}}}\rput(1;202.5){{\footnotesize{3}}}\rput(1;247.5){{\footnotesize{2}}}
\rput(1;292.5){{\footnotesize{1}}}\rput(1;337.5){{\footnotesize{0}}}
\end{pspicture}$ &$242$   & $\{0,1,2,4,6,7\}$ &
$\begin{pspicture}[shift=-1](-1.4,-1.3)(1.4,1.3)
\psline(1;22.5)(1;67.5)(1;112.5)(1;157.5)(1;202.5)(1;247.5)(1;292.5)(1;337.5)(1;22.5)
\psline(1;22.5)(1;112.5)(1;202.5)(1;292.5)(1;22.5)
\psline(1;67.5)(1;157.5)(1;247.5)(1;337.5)(1;67.5)
\psline(1;22.5)(1;202.5)\psline(1;67.5)(1;247.5)\psline(1;112.5)(1;292.5)\psline(1;157.5)(1;337.5)
\pscircle[fillcolor=lightgray, fillstyle=solid, linewidth=1pt](1;22.5){.2}
\pscircle[fillcolor=lightgray, fillstyle=solid, linewidth=1pt](1;67.5){.2}
\pscircle[fillcolor=lightgray, fillstyle=solid, linewidth=1pt](1;112.5){.2}
\pscircle[fillcolor=lightgray, fillstyle=solid, linewidth=1pt](1;157.5){.2}
\pscircle[fillcolor=lightgray, fillstyle=solid, linewidth=1pt](1;202.5){.2}
\pscircle[fillcolor=lightgray, fillstyle=solid, linewidth=1pt](1;247.5){.2}
\pscircle[fillcolor=lightgray, fillstyle=solid, linewidth=1pt](1;292.5){.2}
\pscircle[fillcolor=lightgray, fillstyle=solid, linewidth=1pt](1;337.5){.2}
\rput(1;22.5){{\footnotesize{7}}}\rput(1;67.5){{{\footnotesize{6}}}}\rput(1;112.5){{{\footnotesize{5}}}}
\rput(1;157.5){{\footnotesize{4}}}\rput(1;202.5){{\footnotesize{3}}}\rput(1;247.5){{\footnotesize{2}}}
\rput(1;292.5){{\footnotesize{1}}}\rput(1;337.5){{\footnotesize{0}}}
\end{pspicture}$ &   $112211$  \\
\hline $\{0,3,5\}$ & $\begin{pspicture}[shift=-1](-1.4,-1.3)(1.4,1.3)
\psline(1;22.5)(1;157.5)(1;292.5)(1;67.5)(1;202.5)(1;337.5)(1;112.5)(1;247.5)(1;22.5)
\pscircle[fillcolor=lightgray, fillstyle=solid, linewidth=1pt](1;22.5){.2}
\pscircle[fillcolor=lightgray, fillstyle=solid, linewidth=1pt](1;67.5){.2}
\pscircle[fillcolor=lightgray, fillstyle=solid, linewidth=1pt](1;112.5){.2}
\pscircle[fillcolor=lightgray, fillstyle=solid, linewidth=1pt](1;157.5){.2}
\pscircle[fillcolor=lightgray, fillstyle=solid, linewidth=1pt](1;202.5){.2}
\pscircle[fillcolor=lightgray, fillstyle=solid, linewidth=1pt](1;247.5){.2}
\pscircle[fillcolor=lightgray, fillstyle=solid, linewidth=1pt](1;292.5){.2}
\pscircle[fillcolor=lightgray, fillstyle=solid, linewidth=1pt](1;337.5){.2}
\rput(1;22.5){{\footnotesize{7}}}\rput(1;67.5){{{\footnotesize{6}}}}\rput(1;112.5){{{\footnotesize{5}}}}
\rput(1;157.5){{\footnotesize{4}}}\rput(1;202.5){{\footnotesize{3}}}\rput(1;247.5){{\footnotesize{2}}}
\rput(1;292.5){{\footnotesize{1}}}\rput(1;337.5){{\footnotesize{0}}}
\end{pspicture}$ &  $323$   &  $\{0,1,3,4,5,7\}$  &
$\begin{pspicture}[shift=-1](-1.4,-1.3)(1.4,1.3)
\psline(1;22.5)(1;67.5)(1;112.5)(1;157.5)(1;202.5)(1;247.5)(1;292.5)(1;337.5)(1;22.5)
\psline(1;22.5)(1;157.5)(1;292.5)(1;67.5)(1;202.5)(1;337.5)(1;112.5)(1;247.5)(1;22.5)
\psline(1;22.5)(1;202.5)\psline(1;67.5)(1;247.5)\psline(1;112.5)(1;292.5)\psline(1;157.5)(1;337.5)
\pscircle[fillcolor=lightgray, fillstyle=solid, linewidth=1pt](1;22.5){.2}
\pscircle[fillcolor=lightgray, fillstyle=solid, linewidth=1pt](1;67.5){.2}
\pscircle[fillcolor=lightgray, fillstyle=solid, linewidth=1pt](1;112.5){.2}
\pscircle[fillcolor=lightgray, fillstyle=solid, linewidth=1pt](1;157.5){.2}
\pscircle[fillcolor=lightgray, fillstyle=solid, linewidth=1pt](1;202.5){.2}
\pscircle[fillcolor=lightgray, fillstyle=solid, linewidth=1pt](1;247.5){.2}
\pscircle[fillcolor=lightgray, fillstyle=solid, linewidth=1pt](1;292.5){.2}
\pscircle[fillcolor=lightgray, fillstyle=solid, linewidth=1pt](1;337.5){.2}
\rput(1;22.5){{\footnotesize{7}}}\rput(1;67.5){{{\footnotesize{6}}}}\rput(1;112.5){{{\footnotesize{5}}}}
\rput(1;157.5){{\footnotesize{4}}}\rput(1;202.5){{\footnotesize{3}}}\rput(1;247.5){{\footnotesize{2}}}
\rput(1;292.5){{\footnotesize{1}}}\rput(1;337.5){{\footnotesize{0}}}
\end{pspicture}$ & $(121)^2$   \\
\hline  $\{0,2,4,6\}$ & $\begin{pspicture}[shift=-1](-1.4,-1.3)(1.4,1.3)
\psline(1;22.5)(1;112.5)(1;202.5)(1;292.5)(1;22.5)
\psline(1;67.5)(1;157.5)(1;247.5)(1;337.5)(1;67.5)
\psline(1;22.5)(1;202.5)\psline(1;67.5)(1;247.5)\psline(1;112.5)(1;292.5)\psline(1;157.5)(1;337.5)
\pscircle[fillcolor=lightgray, fillstyle=solid, linewidth=1pt](1;22.5){.2}
\pscircle[fillcolor=lightgray, fillstyle=solid, linewidth=1pt](1;67.5){.2}
\pscircle[fillcolor=lightgray, fillstyle=solid, linewidth=1pt](1;112.5){.2}
\pscircle[fillcolor=lightgray, fillstyle=solid, linewidth=1pt](1;157.5){.2}
\pscircle[fillcolor=lightgray, fillstyle=solid, linewidth=1pt](1;202.5){.2}
\pscircle[fillcolor=lightgray, fillstyle=solid, linewidth=1pt](1;247.5){.2}
\pscircle[fillcolor=lightgray, fillstyle=solid, linewidth=1pt](1;292.5){.2}
\pscircle[fillcolor=lightgray, fillstyle=solid, linewidth=1pt](1;337.5){.2}
\rput(1;22.5){{\footnotesize{7}}}\rput(1;67.5){{{\footnotesize{6}}}}\rput(1;112.5){{{\footnotesize{5}}}}
\rput(1;157.5){{\footnotesize{4}}}\rput(1;202.5){{\footnotesize{3}}}\rput(1;247.5){{\footnotesize{2}}}
\rput(1;292.5){{\footnotesize{1}}}\rput(1;337.5){{\footnotesize{0}}}
\end{pspicture}$ &  $(2)^4$  &
 $\{0,2,3,4,5,6\}$   & $\begin{pspicture}[shift=-1](-1.4,-1.3)(1.4,1.3)
\psline(1;22.5)(1;112.5)(1;202.5)(1;292.5)(1;22.5)
\psline(1;67.5)(1;157.5)(1;247.5)(1;337.5)(1;67.5)
\psline(1;22.5)(1;157.5)(1;292.5)(1;67.5)(1;202.5)(1;337.5)(1;112.5)(1;247.5)(1;22.5)
\psline(1;22.5)(1;202.5)\psline(1;67.5)(1;247.5)\psline(1;112.5)(1;292.5)\psline(1;157.5)(1;337.5)
\pscircle[fillcolor=lightgray, fillstyle=solid, linewidth=1pt](1;22.5){.2}
\pscircle[fillcolor=lightgray, fillstyle=solid, linewidth=1pt](1;67.5){.2}
\pscircle[fillcolor=lightgray, fillstyle=solid, linewidth=1pt](1;112.5){.2}
\pscircle[fillcolor=lightgray, fillstyle=solid, linewidth=1pt](1;157.5){.2}
\pscircle[fillcolor=lightgray, fillstyle=solid, linewidth=1pt](1;202.5){.2}
\pscircle[fillcolor=lightgray, fillstyle=solid, linewidth=1pt](1;247.5){.2}
\pscircle[fillcolor=lightgray, fillstyle=solid, linewidth=1pt](1;292.5){.2}
\pscircle[fillcolor=lightgray, fillstyle=solid, linewidth=1pt](1;337.5){.2}
\rput(1;22.5){{\footnotesize{7}}}\rput(1;67.5){{{\footnotesize{6}}}}\rput(1;112.5){{{\footnotesize{5}}}}
\rput(1;157.5){{\footnotesize{4}}}\rput(1;202.5){{\footnotesize{3}}}\rput(1;247.5){{\footnotesize{2}}}
\rput(1;292.5){{\footnotesize{1}}}\rput(1;337.5){{\footnotesize{0}}}
\end{pspicture}$ & $211112$   \\
\hline $\{0,1,4,7\}$  &  $\begin{pspicture}[shift=-1](-1.4,-1.3)(1.4,1.3)
\psline(1;22.5)(1;67.5)(1;112.5)(1;157.5)(1;202.5)(1;247.5)(1;292.5)(1;337.5)(1;22.5)
\psline(1;22.5)(1;202.5)\psline(1;67.5)(1;247.5)\psline(1;112.5)(1;292.5)\psline(1;157.5)(1;337.5)
\pscircle[fillcolor=lightgray, fillstyle=solid, linewidth=1pt](1;22.5){.2}
\pscircle[fillcolor=lightgray, fillstyle=solid, linewidth=1pt](1;67.5){.2}
\pscircle[fillcolor=lightgray, fillstyle=solid, linewidth=1pt](1;112.5){.2}
\pscircle[fillcolor=lightgray, fillstyle=solid, linewidth=1pt](1;157.5){.2}
\pscircle[fillcolor=lightgray, fillstyle=solid, linewidth=1pt](1;202.5){.2}
\pscircle[fillcolor=lightgray, fillstyle=solid, linewidth=1pt](1;247.5){.2}
\pscircle[fillcolor=lightgray, fillstyle=solid, linewidth=1pt](1;292.5){.2}
\pscircle[fillcolor=lightgray, fillstyle=solid, linewidth=1pt](1;337.5){.2}
\rput(1;22.5){{\footnotesize{7}}}\rput(1;67.5){{{\footnotesize{6}}}}\rput(1;112.5){{{\footnotesize{5}}}}
\rput(1;157.5){{\footnotesize{4}}}\rput(1;202.5){{\footnotesize{3}}}\rput(1;247.5){{\footnotesize{2}}}
\rput(1;292.5){{\footnotesize{1}}}\rput(1;337.5){{\footnotesize{0}}}
\end{pspicture}$ & $1331$   &
$\{0,1,2,3,5,6,7\}$  & $\begin{pspicture}[shift=-1](-1.4,-1.3)(1.4,1.3)
\psline(1;22.5)(1;67.5)(1;112.5)(1;157.5)(1;202.5)(1;247.5)(1;292.5)(1;337.5)(1;22.5)
\psline(1;22.5)(1;112.5)(1;202.5)(1;292.5)(1;22.5)
\psline(1;67.5)(1;157.5)(1;247.5)(1;337.5)(1;67.5)
\psline(1;22.5)(1;157.5)(1;292.5)(1;67.5)(1;202.5)(1;337.5)(1;112.5)(1;247.5)(1;22.5)
\pscircle[fillcolor=lightgray, fillstyle=solid, linewidth=1pt](1;22.5){.2}
\pscircle[fillcolor=lightgray, fillstyle=solid, linewidth=1pt](1;67.5){.2}
\pscircle[fillcolor=lightgray, fillstyle=solid, linewidth=1pt](1;112.5){.2}
\pscircle[fillcolor=lightgray, fillstyle=solid, linewidth=1pt](1;157.5){.2}
\pscircle[fillcolor=lightgray, fillstyle=solid, linewidth=1pt](1;202.5){.2}
\pscircle[fillcolor=lightgray, fillstyle=solid, linewidth=1pt](1;247.5){.2}
\pscircle[fillcolor=lightgray, fillstyle=solid, linewidth=1pt](1;292.5){.2}
\pscircle[fillcolor=lightgray, fillstyle=solid, linewidth=1pt](1;337.5){.2}
\rput(1;22.5){{\footnotesize{7}}}\rput(1;67.5){{{\footnotesize{6}}}}\rput(1;112.5){{{\footnotesize{5}}}}
\rput(1;157.5){{\footnotesize{4}}}\rput(1;202.5){{\footnotesize{3}}}\rput(1;247.5){{\footnotesize{2}}}
\rput(1;292.5){{\footnotesize{1}}}\rput(1;337.5){{\footnotesize{0}}}
\end{pspicture}$ &  $1112111$  \\
\hline  $\{0,3,4,5\}$  &  $\begin{pspicture}[shift=-1](-1.4,-1.3)(1.4,1.3)
\psline(1;22.5)(1;157.5)(1;292.5)(1;67.5)(1;202.5)(1;337.5)(1;112.5)(1;247.5)(1;22.5)
\psline(1;22.5)(1;202.5)\psline(1;67.5)(1;247.5)\psline(1;112.5)(1;292.5)\psline(1;157.5)(1;337.5)
\pscircle[fillcolor=lightgray, fillstyle=solid, linewidth=1pt](1;22.5){.2}
\pscircle[fillcolor=lightgray, fillstyle=solid, linewidth=1pt](1;67.5){.2}
\pscircle[fillcolor=lightgray, fillstyle=solid, linewidth=1pt](1;112.5){.2}
\pscircle[fillcolor=lightgray, fillstyle=solid, linewidth=1pt](1;157.5){.2}
\pscircle[fillcolor=lightgray, fillstyle=solid, linewidth=1pt](1;202.5){.2}
\pscircle[fillcolor=lightgray, fillstyle=solid, linewidth=1pt](1;247.5){.2}
\pscircle[fillcolor=lightgray, fillstyle=solid, linewidth=1pt](1;292.5){.2}
\pscircle[fillcolor=lightgray, fillstyle=solid, linewidth=1pt](1;337.5){.2}
\rput(1;22.5){{\footnotesize{7}}}\rput(1;67.5){{{\footnotesize{6}}}}\rput(1;112.5){{{\footnotesize{5}}}}
\rput(1;157.5){{\footnotesize{4}}}\rput(1;202.5){{\footnotesize{3}}}\rput(1;247.5){{\footnotesize{2}}}
\rput(1;292.5){{\footnotesize{1}}}\rput(1;337.5){{\footnotesize{0}}}
\end{pspicture}$ & $3113$   &
$\{0,1,2,3,4,5,6,7\}$ & $\begin{pspicture}[shift=-1](-1.4,-1.3)(1.4,1.3)
\psline(1;22.5)(1;67.5)(1;112.5)(1;157.5)(1;202.5)(1;247.5)(1;292.5)(1;337.5)(1;22.5)
\psline(1;22.5)(1;112.5)(1;202.5)(1;292.5)(1;22.5)
\psline(1;67.5)(1;157.5)(1;247.5)(1;337.5)(1;67.5)
\psline(1;22.5)(1;157.5)(1;292.5)(1;67.5)(1;202.5)(1;337.5)(1;112.5)(1;247.5)(1;22.5)
\psline(1;22.5)(1;202.5)\psline(1;67.5)(1;247.5)\psline(1;112.5)(1;292.5)\psline(1;157.5)(1;337.5)
\pscircle[fillcolor=lightgray, fillstyle=solid, linewidth=1pt](1;22.5){.2}
\pscircle[fillcolor=lightgray, fillstyle=solid, linewidth=1pt](1;67.5){.2}
\pscircle[fillcolor=lightgray, fillstyle=solid, linewidth=1pt](1;112.5){.2}
\pscircle[fillcolor=lightgray, fillstyle=solid, linewidth=1pt](1;157.5){.2}
\pscircle[fillcolor=lightgray, fillstyle=solid, linewidth=1pt](1;202.5){.2}
\pscircle[fillcolor=lightgray, fillstyle=solid, linewidth=1pt](1;247.5){.2}
\pscircle[fillcolor=lightgray, fillstyle=solid, linewidth=1pt](1;292.5){.2}
\pscircle[fillcolor=lightgray, fillstyle=solid, linewidth=1pt](1;337.5){.2}
\rput(1;22.5){{\footnotesize{7}}}\rput(1;67.5){{{\footnotesize{6}}}}\rput(1;112.5){{{\footnotesize{5}}}}
\rput(1;157.5){{\footnotesize{4}}}\rput(1;202.5){{\footnotesize{3}}}\rput(1;247.5){{\footnotesize{2}}}
\rput(1;292.5){{\footnotesize{1}}}\rput(1;337.5){{\footnotesize{0}}}
\end{pspicture}$ &  $(1)^8$  \\
\hline
\end{tabular}
\vskip .2cm
\caption{ The circulant graphs of order $8$ and the corresponding palindromes} \label{pal8exafig}
\end{table}
There are four disconnected circulant graphs out of sixteen circulant graphs of order $8$, where their
generating sets are $\{0\}$, $\{0,4\}$, $\{0,2,6\}$ and $\{0,2,4,6\}$, see Table \ref{pal8exafig}. On the other hand, there are four periodic
palindromes $(4)^2$, $(2)^4$, $(121)^2$ and $(1)^8$. One can see that the bijection in Theorem~\ref{palthm} does not preserve
the connectedness of the circulant graph $G(n, \Omega)$
to the aperiodicity of the corresponding palindromes $\sigma_{\Omega}$ as we have seen in Example~\ref{aperiodic8exa}.\\

For each integer $n\geq 2$, we now define
\[P_A(n)=\{\sigma \in  P(n) \mid \sigma \mbox{~is~aperiodic}\}
\mbox{~and~} \mathcal{G}(n)=\{\Omega \in \mathcal{G}_P(n) \mid  \langle \Omega \rangle =\BZ_n \}.\]

In the following lemma, we will show
how $\gcd(\sigma_{\Omega})$ classifies the connectedness of the corresponding circulant graphs
$G(n, \Omega)$, and the result will be used to prove Theorem~\ref{apn-thm}.

\begin{lem}\label{generating-lem}
For each integer $n\geq 2$, the following are equivalent.
\begin{enumerate}
\item[(i)] The circulant graph $G(n, \Omega)$ is connected.
\item[(ii)] $\Omega\in \mathcal{G}(n)$.
\item[(iii)] $\gcd (\sigma_{\Omega})=1$.
\end{enumerate}
\end{lem}
\begin{proof}
Note that for each $\Omega \in \mathcal{G}_P(n)$, $\Omega$ is a
generating set for group $\BZ_n$ if and only if $\gcd (\Omega)=1$. By $\gcd (\Omega)=\gcd (\sigma_{\Omega})$, the result now follows immediately.
\end{proof}

To define a map between $P_{A}(n)$ and $\mathcal{G}(n)$ in Theorem \ref{apn-thm}, we need to define the following notation.\\
For any periodic palindrome of the form $\sigma=\overbrace{(c_1\cdots c_k)\cdots
(c_1\cdots c_k)}^\text{$r$ times of ($c_1\cdots c_k$)}$ ($i.e.$, $r$ times repetition of aperiodic palindrome 
$c_1\cdots c_k$), we denote it by $\sigma=(c_1\cdots c_k)^r$. For any aperiodic palindromes of the form $\sigma'=\sigma_1\sigma_2\cdots \sigma_m$, if $\gcd(\sigma')=d\neq 1$ then we define a periodic composition $\nu(\sigma'):= (\frac{\sigma_1}{d} \frac{\sigma_2}{d} \cdots \frac{\sigma_m}{d})^d$.\\

In Example~\ref{aperiodic8exa}, among twelve aperiodic palindromes
of $8$, two decompositions $242$, $8$ have $\gcd(242)=2$ and $\gcd(8)=8$ and thus $\nu(242)=(121)^2$ and $\nu(8)=(1)^8$. \\

\begin{thm} \label{apn-thm}
For each integer $n\geq 2$, there is a one-to-one correspondence between the set of aperiodic palindromes of $n$ and the set of connected circulant graphs of order $n$.
\end{thm}
\begin{proof}
Define a map $\tau : P_A(n) \longrightarrow \mathcal{G}(n)$ by
\begin{equation}\label{def-2ndftn} \nonumber
\tau(\sigma)= \begin{cases}
(\psi|_{\mathcal{G}_P(n)})^{-1}(\sigma)=\Omega_{\sigma} & {\rm{if}}~\gcd(\sigma)=1\\
(\psi|_{\mathcal{G}_P(n)})^{-1}(\nu(\sigma))=\Omega_{\nu(\sigma)} & {\rm{if}}~\gcd(\sigma)=d\neq 1
\end{cases}
\end{equation}
where $\nu(\sigma)= (\frac{\sigma_1}{d} \frac{\sigma_2}{d} \cdots \frac{\sigma_m}{d})^d$ and $\psi|_{\mathcal{G}_P(n)}$ is the bijective map defined in Theorem \ref{palthm}. Then the map $\tau$ is well-defined as $\tau(\sigma)\in \mathcal{G}_P(n)$ by Theorem \ref{palthm} and $ \langle \Omega \rangle =\BZ_n$ by Lemma~\ref{generating-lem} for $\gcd(\nu(\sigma))=1$. As the map $\tau$ is injective by Theorem \ref{palthm}, to prove the result, it is enough to show that $\tau$ is surjective. Take an element $\Omega=\{a_1, a_2, \ldots, a_t\} \in \mathcal{G}(n)$. If $\psi|_{\mathcal{G}_P(n)}(\Omega) \in P_A(n)$ then it is easy to see that $\tau(\psi|_{\mathcal{G}_P(n)}(\Omega))=\Omega$. Now suppose $\psi|_{\mathcal{G}_P(n)}(\Omega) \not\in P_A(n)$. Then $\psi|_{\mathcal{G}_P(n)}(\Omega)$ is a periodic palindrome and $\gcd(\psi|_{\mathcal{G}_P(n)}(\Omega))=1$ by Lemma \ref{generating-lem}. Thus put $\psi|_{\mathcal{G}_P(n)}(\Omega)=(\sigma_1 \sigma_2 \cdots \sigma_m)^r$ for some $m\geq 1$ and $r \ge 2$, where palindrome $\sigma_1 \sigma_2 \cdots \sigma_m$ is aperiodic and $\gcd(\psi|_{\mathcal{G}_P(n)}(\Omega))=\gcd \{\sigma_1, \sigma_2, \ldots, \sigma_m\}=1$. Then palindrome $(r\sigma_1) (r\sigma_2) \cdots (r\sigma_m)$ satisfies $(r\sigma_1) (r\sigma_2) \cdots (r\sigma_m)\in P_{A}(n)$ and $\tau( (r\sigma_1) (r\sigma_2) \cdots (r\sigma_m))=\Omega_{\nu((r\sigma_1) (r\sigma_2) \cdots (r\sigma_m))}
=\Omega_{(\sigma_1 \sigma_2 \cdots \sigma_m)^r}=\Omega$, and thus the map $\tau$ is surjective. This completes the proof.
\end{proof}

By combining Theorem~\ref{apn-thm} and \cite[Corollary~3.2]{KKL}, the following corollary follows immediately.

\begin{cor} \label{acounting}
For each integer $n\geq 2$, the number of aperiodic palindromes of $n$ is
$$\sum_{d|n} \mu\left(\dfrac{n}{d}\right) \left(2^{\lfloor \frac{d}{2} \rfloor}
-1 \right)$$
where $\mu$ is the M\"obius function.

\end{cor}

\begin{remark}
One may find the number of aperiodic palindromes of $n$ with $n \leq 127$ in~\cite{baek}. This result improves the previously known result $n\le 55$ in \cite{OEIS}.
\end{remark}


\begin{thebibliography}{9}


\bibitem{AP} B. Alspach and T. D. Parsons, \textit{Isomorphisms of circulant graphs and digraphs}, Discrete Math. 25 (1979) 97–-108.

\bibitem{baek} H. Baek, \textit{The number of aperiodic palindromes of $n \le 127$}, available at {\tt http://chaos.cu.ac.kr/~tnad/}.

\bibitem{Dobson} E. Dobson, \textit{On isomorphisms of circulant digraphs of bounded degree}, Discrete Math. 308 (2008) 6047–-6055.

\bibitem{HCG} S. Heuhach, P. Chinn and R. Grimaldi, \textit{Rises, Levels, Drops and "+" Signs in Compositions: Extensions of a
Paper by Alladi and Hoggatt}, Fibonacci Quarterly 41(3) (1975) 229--239.

\bibitem{KKL} D. Kim, Y. Kwon and J. Lee, \textit{Degree distributions for a class of Circulant graphs},
preprint.

\bibitem{Muzy2} M. Muzychuk, \textit{A solution of the isomorphism problem for circulant graphs}, Proc. London Math. Soc. 88 (2004) 1--41.

\bibitem{OEIS} N. J. A. Sloane, \textit{The On-Line Encyclopedia of Integer Sequences}, {\tt http://oeis.org/A179519}.

\bibitem{wiki} Wikipedia, \textit{Palindromic number},  {\tt{http://en.wikipedia.org/wiki/Palindromic}$\underline{\mbox{ }}$\tt{number}}.


\end{thebibliography}
\end{document}